\documentclass[11pt]
{amsart}
\usepackage{amssymb,amsmath,amsthm,amsfonts,amsopn,url,color,hyperref,
mathtools,microtype,MnSymbol,dutchcal,
mathrsfs}
\usepackage{scalerel}
\usepackage{stackengine,wasysym}
\usepackage[shortlabels]{enumitem}

\usepackage[normalem]{ulem}
\usepackage[all]{xy}
\input xy
\xyoption{all}
\usepackage{amscd}
\usepackage{soul}

\theoremstyle{plain}
\newtheorem{thm}{Theorem}[section]
\newtheorem{mth}[thm]{Metatheorem}

\newtheorem{prop}[thm]{Proposition}
\newtheorem{lemma}[thm]{Lemma}
\newtheorem{cor}[thm]{Corollary}

\newtheorem{defpr}[thm]{Definition/Proposition}
\newtheorem{thmdef}[thm]{Theorem/Definition}
\theoremstyle{definition}
\newtheorem{defn}[thm]{Definition}
\newtheorem*{defn*}{Definition}
\newtheorem*{question*}{Question}

\newtheorem{example}[thm]{Example}
\newtheorem*{example*}{Example}
\newtheorem{rem}[thm]{Remark}
\newtheorem*{rem*}{Remark}

\newcommand{\field}[1]{\mathbb{#1}}
\newcommand{\N}{\field{N}}
\newcommand{\Z}{\field{Z}}

\newcommand{\F}{\field{F}}

\newcommand{\ideal}[1]{\mathfrak{#1}}
\newcommand{\m}{\ideal{m}}
\newcommand{\n}{\ideal{n}}
\newcommand{\p}{\ideal{p}}
\newcommand{\q}{\ideal{q}}
\newcommand{\ia}{\ideal{a}}
\newcommand{\ib}{\ideal{b}}

\newcommand{\func}[1]{\mathrm{#1} \,}

\newcommand{\Spec}{\func{Spec}}

\newcommand{\hgt}{\func{ht}}

\newcommand{\ra}{\rightarrow}

\newcommand{\be}{\begin{enumerate}}
\newcommand{\ee}{\end{enumerate}}

\newcommand{\li}
 {\leftfootline}

\newcommand{\into}{\hookrightarrow}

\newcommand{\cF}{\mathcal{F}}
\newcommand{\cI}{\mathcal{I}}
\newcommand{\cJ}{\mathcal{J}}
\newcommand{\cK}{\mathcal{K}}
\newcommand{\cL}{\mathcal{L}}
\newcommand{\cO}{\mathcal{O}}
\newcommand{\cM}{\mathcal{M}}

\newcommand{\cb}{\mathcal{b}}

\renewcommand{\phi}{\varphi}

\newcommand{\cl}{{\mathrm{cl}}}
\let\int\relax
\DeclareMathOperator{\int}{i}

\newcommand{\sM}{\mathscr{M}}
\newcommand{\sR}{\mathscr{R}}

\DeclareMathOperator{\Img}{Im}
\DeclareMathOperator{\Idl}{Id}

\DeclareMathOperator{\Sub}{Sub}

\author{Neil Epstein}
\address{Department of Mathematical Sciences \\ George Mason University \\ Fairfax, VA  22030}
\email{nepstei2@gmu.edu}

\title{Tight closure, coherence, and localization at single elements}

\date{February 29, 2024}
\begin{document}
\begin{abstract}
In this note, a condition (\emph{open persistence}) is presented under which a (pre)closure operation on submodules (resp. ideals) over rings of global sections over a scheme $X$ can be extended to a (pre)closure operation on sheaves of submodules of a coherent $\cO_X$-module (resp. sheaves of ideals in $\cO_X$).  A second condition (\emph{glueability}) is given for such an operation to behave nicely.  It is shown that for an operation that satisfies both conditions, the question of whether the operation commutes with localization at single elements is equivalent to the question of whether the new operation preserves quasi-coherence.  It is shown that both conditions hold for tight closure and some of its important variants, thus yielding a geometric reframing of the open question of whether tight closure localizes at single elements.  A new singularity type (\emph{semi F-regularity}) arises, which sits between F-regularity and weak F-regularity.  The paper ends with (1) a case where semi F-regularity and weak F-regularity coincide, and (2) a case where they cannot coincide without implying a solution to a major conjecture.
\end{abstract}
\maketitle

Tight closure theory 
has long been known to be connected with algebraic geometry -- e.g. in the connection between test ideals and multiplier ideals \cite{Sm-utest}, theorems about line bundles \cite{SmFuj+}, the connection between F-singularities and singularities arising from resolution of singularities \cite{SmFrat, Harat}, etc.. Hence, it was an early goal to establish that the tight closure operation, like the integral closure operation, would commute with localization at arbitrary multiplicatively closed sets \cite{HH-sFreg}.  Further pressure came from the theorem that if this were true, then tight closure would commute with arbitrary regular base change in all cases of geometric interest \cite{HHbase}. The question lay open for some 20 years until it was answered in the negative \cite{BM-unloc}.  Since then, the general attitude in the commutative algebra community has been that the tight closure operation itself is ultimately nongeometric, and that the \emph{true} geometry that comes from tight closure theory lies in artifacts that arise from its study (e.g. test ideals, the notion of strong F-regularity, etc.).

In this work, we take a different tack.  Namely, even though tight closure does not exhibit \emph{all} of the localization behavior one desires, it does behave well with respect to localization in some ways, and this can be exploited to create notions of tight closure for quasi-coherent sheaves over an open subset of a Noetherian scheme of prime characteristic.  Arising from these ideas, we obtain a singularity type \emph{between} weak F-regularity and F-regularity, and we recast the question of whether tight closure commutes with localization at a \emph{single} element of a ring in terms of a question about the coherence of certain sheaves of ideals or submodules.

Our approach is quite general.  Indeed, we isolate two characteristics of a closure operation, or even a preclosure operation, on ideals or submodules that allow for a sheaf theory based on that operation.  If a (pre)closure operation is \emph{openly persistent} (meaning the operation is compatible with restriction to affine open subsets of an affine scheme), then a natural extension from a (pre)closure operation ideals or submodules to sheaves of ideals or of submodules is constructed.  In fact, it is shown that the resulting operation on sheaves of ideals, or on sheaves of submodules of a given quasi-coherent module, is itself a (pre)closure operation on the corresponding poset.  Interestingly, though the ambient sheaf must be quasi-coherent, the subsheaf need not be quasi-coherent in order to construct the closure sheaf.  This is important because we do not know whether the closure of a coherent sheaf of ideals must be quasi-coherent.

If in addition a (pre)closure operation is \emph{glueable} (a technical condition), and the sheaf of ideals or submodules is quasi-coherent, then the new operation becomes easier to work with than otherwise, as one has attractive equivalent definitions for it.  Quasi-coherence becomes a key question, as it is shown that for a glueable preclosure operation on ideals or submodules, commutation with localization at single elements is \emph{equivalent} to quasi-coherence being maintained when the resulting operation is applied to a quasi-coherent subsheaf.  We then show that any preclosure operation that commutes with localization at single elements, as well as tight closure and its variants like $\ia^t$-tight closure, are openly persistent and glueable.  Hence all of the above results apply, and the single-element tight closure localization question is recast as a question about coherence of tight closure sheaves.

We obtain a new F-singularity type \emph{between} that of F-regularity and weak F-regularity, which we dub \emph{semi F-regularity}.  It is defined by saying that every sheaf of ideals is tightly closed.  We show that it is equivalent to weak F-regularity in Jacobson rings, but that they cannot be equivalent in local rings without implying that F-regularity is the same as weak F-regularity.

\section{Persistence and closure operations on subsheaves}

The purpose of this section is to give conditions under which a (pre)closure operation on ideals or submodules over rings arising from a scheme lifts to a (pre)closure operation on sheaves of ideals or of submodules over the scheme.  For unexplained terminology in algebraic geometry, see \cite{Hart-AG}.  For background on closure operations, see the survey \cite{nme-guide2} or the book \cite{El-clbook}.

\begin{defn}
Let $(S,\leq)$ be a partially ordered set.  A \emph{preclosure operation} 
on $S$ is a function $\cl: S \ra S$ (written $s \mapsto s^\cl$) that is:
\begin{itemize}
    \item \emph{extensive} (i.e., $s \leq s^\cl$ for all $s\in S$) and
    \item\emph{order-preserving} (i.e., whenever $s,t\in S$ with $s \leq t$, we have $s^\cl \leq t^\cl$).
\end{itemize}
A preclosure operation $\cl$ on $S$ is a \emph{closure operation} on $S$ if it is also \begin{itemize}
    \item \emph{idempotent} (i.e. $(s^\cl)^\cl = s^\cl$ for all $s\in S$).
\end{itemize}

Let $R$ be a ring and $M$ an $R$-module.  Set $(\Idl(R), \subseteq)$ to be the poset of ideals of $R$.  Then a \emph{(pre)closure operation on (the ideals of) $R$} is a (pre)closure operation on the poset $\Idl(R)$.  Let $(\Sub(M), \subseteq)$ be the poset of $R$-submodules of $M$.  Then a \emph{(pre)closure operation on (the submodules of) $M$} is a (pre)closure operation on $\Sub(M)$.

If $\sR$ is a class of rings and $\sM$ is a class of modules over various members of $\sR$, and we have a definition of a (pre)closure operation $\cl_M$ on each $M \in \sM$, we say $\cl$ is \emph{defined on (submodules in) $\sM$}.  If $R\in \sR$ and $M \in \sM$ is an $R$-module, then for a submodule $L$ of $M$, we write $L^\cl_M$ for $L^{\cl_M}$.
\end{defn}

If $L \subseteq M$ are $R$-modules and $R \ra S$ is a ring map, we write $LS$ for the image of the $S$-module homomorphism $L \otimes_R S \ra M \otimes_R S$ induced by the $R$-linear inclusion map $ L \hookrightarrow M$.

\begin{defn}
Let $R$ be a ring.  Let $\cl$ be a preclosure operation, such that for any open affine subset $U=\Spec S \subseteq \Spec R$, $\cl$ is defined on the ideals of $S$ (resp. the submodules of any (finite) $S$-module). Then we say $\cl$ is \emph{openly persistent} on ideals (resp. (finite) modules) over $\Spec R$ if for any ideal $I$ of $R$ and $x \in I^\cl$ (resp. any $R$-submodule inclusion $L \subseteq M$ (with $M$ finite) and $x\in L^\cl_M$), and any affine open subset $U \subseteq \Spec R$ with corresponding ring map $\phi: R \ra S$, we have $\phi(x) \in (IS)^\cl$ (resp. $x \otimes 1 \in (LS)^\cl_{M \otimes_R S}$).

More generally, for a scheme $X$, we say $\cl$ is \emph{openly persistent on ideals (resp. (finite) modules) over $X$} if for any open affine $U \subseteq X$, $\cl$ is openly persistent on ideals (resp. (finite) modules) over $U$.
\end{defn}

\begin{rem}\label{rem:flatly}
    Suppose $\cl$ is a closure operation that is persistent over flat ring maps.  Then it is openly persistent.  This is because the ring homomorphisms $R \ra S$ arising from open immersions are always flat \cite[Proposition III.9.2(a)]{Hart-AG}.
\end{rem}

\begin{rem} If $\cl$ is openly persistent over $\Spec R$, then for any $f\in R$, we always have $I^\cl R_f \subseteq (IR_f)^\cl$ (resp. $(L^\cl_M)_f \subseteq (L_f)^\cl_{M_f}$).  This is because the localization map $R \ra R_f$ corresponds to the open immersion $D(f) \hookrightarrow \Spec R$.
\end{rem}

\begin{example}\label{ex:Houston}
Open persistence does not always imply \emph{commuting} with localization, even at multiplicative sets generated by a single element.  To see this, first recall the $v$- and $t$-operations.  Namely, if $R$ is an integral domain with fraction field $F$, then for any fractional ideal $\ia$ of $R$, one defines $\ia^{-1} := \{x \in F \mid x\ia \subseteq R$.  Then we set $(0)^v = (0)$, and for a nonzero ideal $I$, we set $I^v := (I^{-1})^{-1}$. If $I=I^v$, we say $I$ is \emph{divisorial} (e.g., any principal ideal). We let $I^t :=\bigcup\{J^v \mid J \subseteq I$, $J$ is a finitely generated ideal$\}.$  Then both $v$ and $t$ are closure operations, with $t\leq v$ (see \cite[p. 15]{El-clbook}).  Clearly if $R$ is Noetherian, then $t=v$.  Note also that whenever $I$ is an ideal of a ring $R$ and $R \ra S$ is flat, then $I^tS \subseteq (IS)^t$ \cite[Proposition 1.2.3 (1) $\implies$ (5)]{El-clbook}.  Therefore, $t$ is openly persistent by Remark~\ref{rem:flatly}.

Now let $T = \mathbb Q[x,y]_{(x,y)}$, let $\m$ be the maximal ideal of $T$, let $p$ be a positive prime number, and consider the subring $R = \mathbb Z_{p\Z} + \m$ of $T$.  Note that $R_p = R[1/p] = T$.  Also, since $\m$ is a maximal ideal of $T$, $\m T = \m \subseteq R$, and $T \neq R$, it follows that $\m = (R :T)$.  Thus by \cite[Proposition 6]{Ba-stdiv}, $\m$ is divisorial as an ideal of $R$.  Hence it is $t$-closed.
On the other hand, by \cite[Exercise 12.4]{Mats}, since $T$ is a Krull domain, any proper nonzero divisorial ideal of $T$ must have height 1, and $\hgt \m = 2$, so $\m$ is not divisorial in $T$ (hence also not $t$-closed since $T$ is Noetherian).  Thus, $(\m R)^t R_p \subsetneq (\m R_p)^t$.

This example is due to Evan Houston.
\end{example}

\begin{lemma}\label{lem:openpers}
Let $\cl$ be an openly persistent preclosure operation on modules over a scheme $X$.  
Let $\cM$ be a quasi-coherent $\cO_X$-module, and $\cL$ a sheaf of $\cO_X$-submodules of $\cM$.  Let $U$, $V$ be affine open subsets of $X$ with $V \subseteq U$.  Then for any $s\in \cL(U)^\cl_{\cM(U)}$, we have $s\vert_V \in \cL(V)^\cl_{\cM(V)}$.
\end{lemma}

\begin{proof}
It is harmless to assume $X=U$.  Then let $A \ra B$ be the ring map representing the open immersion $V \subseteq X$.  Let $M = \cM(X)$ (so that $\tilde{M} = \cM$) and $L = \cL(X)$, so that $s \in L^\cl_M$.  By open persistence of $\cl$, we have $s \otimes 1 \in (LB)^\cl_{M \otimes_A B}$.  Now, consider the following commutative diagram of $B$-modules. \[
\xymatrix{
L \otimes_A B \ar^\psi[r]  \ar@{^{(}->}[d]^i & \cL(V) \ar@{^{(}->}[d]^j\\
M \otimes_A B \ar^{\cong}_\phi[r]& \cM(V)
}
\]
Note that $\phi$ is an isomorphism because $\cM$ is quasi-coherent \cite[Chapter 5, Exercise 1.4]{Liu-AGbook}, and $i$ is injective since $B$ is flat over $A$. The fact that $j$ is injective is part of what ``sheaf of submodules'' means.  It follows that $\psi$ is injective, and that the image of $\psi$ can be identified with $LB = (L \otimes_A B \ra M \otimes_A B)= \Img i$.  That is, we have $LB \subseteq \cL(V) \subseteq \cM(V) = M \otimes_A B$, so $s\vert_V = s \otimes 1 \in (LB)^*_{M \otimes_A B} \subseteq \cL(V)^*_{\cM(V)}$ under these identifications.
\end{proof}

\begin{thmdef}\label{thmdef:tcmod}
Let $\cl$ be an openly persistent preclosure operation on modules over a scheme $X$.  Then we can extend it to a preclosure operation on subsheaves of quasi-coherent sheaves of modules.

In particular, let $\cM$ be a quasi-coherent $\cO_X$-module, and let $\cL\subseteq \cM$ be a sheaf of submodules.  Define a sheaf $\cL^\cl_\cM$ of submodules of $\cM$ as follows: For $U \subseteq X$ open and $s \in \cM(U)$, we say $s \in \cL^\cl_\cM(U)$ if for any $x\in U$, there is an open affine set $V$ with $x \in V \subseteq U$ such that $s\vert_V \in \cL(V)^\cl_{\cM(V)}$ as $\cO_X(V)$-modules.

Given $s\in \cM(U)$, we have $s\in \cL^\cl_\cM(U)$ if and only if there is some collection $\{U_\alpha\}_{\alpha \in A}$ of open affine subsets of $U$ with $U = \bigcup_{\alpha \in A} U_\alpha$ such that $s\vert_{U_\alpha} \in \cL(U_\alpha)^\cl_{\cM(U_\alpha)}$ for all $\alpha \in A$.

The above definition makes $\cL^\cl_\cM$ a subsheaf of $\cO_X$-modules of $\cM$.  Moreover, this is a preclosure operation on subsheaves of $\cM$, in that \begin{itemize}
    \item For any subsheaf $\cL$ of $\cM$, $\cL \subseteq \cL^\cl_\cM$, and
    \item Given subsheaf inclusions $\cK\subseteq \cL \subseteq \cM$, then $\cK^\cl_\cM$ is a subsheaf of $\cL^\cl_\cM$,
\end{itemize}
If, moreover, $X$ is Noetherian, $\cM$ is coherent, and $\cl$ is a closure operation (i.e. idempotent), then we obtain a closure operation on subsheaves on $\cM$, in that we also have idempotence.  That is:
\begin{itemize}
    \item For any subsheaf $\cL$ of $\cM$, we have $\left(\cL^\cl_\cM\right)^\cl_\cM = \cL^\cl_\cM$.
\end{itemize}
\end{thmdef}

\begin{proof}
First we prove that for any open $U \subseteq X$,  $\cL^\cl_\cM(U)$ is an $\cO_X(U)$-module. To see this, let $s,t \in \cL^\cl_\cM(U)$ and $a \in \cO_X(U)$.  Let $x\in U$ and let $V,V'$ affine open sets with $x \in V \subseteq U$, $x \in V' \subseteq U$ such that $s\vert_V \in \cL(V)^\cl_{\cM(V)}$ and $t\vert_{V'} \in \cL(V')^\cl_{\cM(V')}$.  Let $W$ be an affine open subset of $V \cap V'$ with $x\in W$. (Such $W$ exists because $X$ has a basis of affine open subsets.)  Then by Lemma~\ref{lem:openpers}, we have, $s\vert_W, t\vert_W \in \cL(W)^\cl_{\cM(W)}$.  Then since $\cL(W)^\cl_{\cM(W)}$ is an $\cO_X(W)$-module, we have  $(as+t)\vert_W = (a\vert_W \cdot s\vert_W) + t\vert_W \in \cL(W)^\cl_{\cM(W)}$.  Thus by definition, $as+t \in \cL^\cl_\cM(U)$, so that $\cL^\cl_\cM(U)$ is an $\cO_X(U)$-module.

Next we prove the restriction axiom for $\cL^\cl_\cM$.  Let $V \subseteq U$ be open subsets of $X$ and let $s\in \cL^\cl_\cM(U)$.  Let $x\in V$.  Then there is some affine open $W$ with $x\in W \subseteq U$ such that $s\vert_W \in \cL(W)^\cl_{\cM(W)}$.  Let $C$ be an affine open subset of $V \cap W$ with $x\in C$.  Then by Lemma~\ref{lem:openpers}, $(s\vert_V)\vert_C = s\vert_C = (s\vert_W)\vert_C \in \cL(C)^\cl_{\cM(C)}$.  Thus, $s\vert_V \in \cL^\cl_\cM(V)$.

The locality property for $\cL^\cl_\cM$ follows from that of $\cM$.


To see the equivalence in the two definitions, let $U$ be an open set and $s\in \cM(U)$.  First suppose $s\in \cL^\cl_\cM(U)$.  Then for each $x\in U$, there is some open affine $V_x$ with $x\in V_x \subseteq U$ and $s\vert_{V_x} \in \cL(V_x)^\cl_{\cM(V_x)}$.  Thus, the collection $\{V_x \mid x\in U\}$ serves as the required open cover.  Conversely, suppose there is an affine open cover $\{U_\alpha\}_{\alpha \in A}$ of $U$ such that $s\vert_{U_\alpha} \in \cL(U_\alpha)^\cl_{\cM(U_\alpha)}$ for all $\alpha \in A$.  Then for any $x\in U$, there is some $\alpha$ with $x \in U_\alpha \subseteq U$, whence $s \in \cL^\cl_\cM(U)$.

Next we prove the gluing axiom for $\cL^\cl_\cM$.  Let $U$ be an open subset of $X$.  Let  $\{U_\alpha\}_{\alpha \in A}$ be an open over of $U$. Let $s\in \cM(U)$ such that for each $\alpha \in A$, we have $s\vert_{U_\alpha} \in \cL^\cl_\cM(U_\alpha)$.  We need to show that $s\in \cL^\cl_\cM(U)$.  To see this, let $x\in U$.  Then there is some $\alpha$ with $x\in U_\alpha$.  Since $s\vert_{U_\alpha} \in \cL^\cl_\cM(U_\alpha)$, there is some affine open $W$ with $x\in W \subseteq U_\alpha$ such that $s\vert_W = (s\vert_{U_\alpha})\vert_W \in \cL(W)^\cl_{\cM(W)}$.  Since $W \subseteq U$ and $x\in U$ was arbitrary, it follows that $s \in \cL^\cl_\cM(U)$.

Finally we prove that $(-)^\cl_\cM$ is a (pre)closure operation on sheaves of  $\cO_X$-submodules of $\cM$:

Let $\cL$ be a sheaf of $\cO_X$-submodules of $\cM$, $U$ an open subset of $X$, and $s \in \cL(U)$.  Let $V$ be any open affine subset of $U$ with $s \in V \subseteq U$.  Then $s\vert_V \in \cL(V) \subseteq \cL(V)^\cl_{\cM(V)}$, since $\cl$ is a preclosure operation on submodules of the  $\cO_X(V)$-module $\cM(V)$.  Hence, $s \in \cL^\cl_\cM(U)$.

Now let $\cK \subseteq \cL$ be sheaves of submodules of $\cM$. Let $U$ be an open subset of $X$ and $s \in \cK^\cl_\cM(U)$.  Let $x\in U$.  Then there is some open affine $V$ with $x\in V \subseteq U$ and $s\vert_V \in \cK(V)^\cl_{\cM(V)}$.  But since $\cl$ is a preclosure operation on submodules of the  $\cO_X(V)$-module $\cM(V)$, and $\cK(V)$ is an $\cO_X(V)$-submodule of $\cL(V)$, it follows that $s\vert_V \in \cL(V)^\cl_{\cM(V)}$.  Thus, $s \in \cL^\cl_\cM(U)$.  Since $U$ was arbitrary, it follows that $\cK^\cl_\cM$ is a subsheaf of $\cL^\cl_\cM$.

Finally, let $\cL$ be a sheaf of $\cO_X$-submodules of $\cM$, where $X$ is Noetherian, $\cM$ is coherent, and $\cl$ is idempotent on submodules.  Let $U$ be an open subset of $X$, and let $s \in \left(\cL^\cl_\cM\right)^\cl_\cM(U)$.  Let $x\in U$.  Then there is an open affine $V$ with $x\in V \subseteq U$ and $s\vert_V \in \cL^\cl_\cM (V)^\cl_{\cM(V)}$.  Since $\cO_X(V)$ is Noetherian and $\cM(V)$ is finitely generated as an $\cO_X(V)$-module, it follows that the submodule $\cL^\cl_\cM (V)$ is also finitely generated over $\cO_X(V)$.  Say $t_1, \ldots, t_n \in \cL^\cl_\cM(V)$ comprise a generating set.  Then for each $1\leq i \leq n$, there is an open affine $V_i$ with $x\in V_i \subseteq V$ such that ${t_i}\vert_{V_i} \in \cL(V_i)^\cl_{\cM(V_i)}$. Now let $W$ be an open affine with $x \in W \subseteq \bigcap_{i=1}^n V_i$.  Then by Lemma~\ref{lem:openpers}, each $t_i\vert_W \in \cL(W)^\cl_{\cM(W)}$.  Thus, $\cL^\cl_\cM(V)\vert_W \subseteq \cL(W)^\cl_{\cM(W)}$.  Then by open persistence of $\cl$, we have $s\vert_W \in \left(\cL^\cl_\cM(V)\vert_W\right)^\cl_{\cM(W)} \subseteq \left(\cL(W)^\cl_{\cM(W)}\right)^\cl_{\cM(W)} = \cL(W)^\cl_{\cM(W)}$, where both the containment and the last equality follow from the fact that $\cl$ is a closure operation on $\cO_X(W)$-submodules of $\cM(W)$.  Since $x\in U$ was arbitrary and $W$ is an open affine with $x \in W \subseteq U$, it follows that $s \in \cL^\cl_\cM(U)$.  Hence, $\left(\cL^\cl_\cM\right)^\cl_\cM$ is a subsheaf of $\cL^\cl_\cM$.  But we already know the subsheaf relation holds in the other direction.  Hence, $\left(\cL^\cl_\cM\right)^\cl_\cM = \cL^\cl_\cM$.
\end{proof}

\begin{rem}
Suppose one only had a (pre)closure operation on ideals, and not submodules, then the above holds for sheaves of ideals -- with the same proof, restricting our attention to $\cM = \cO_X$.  That is, we have the following:
\end{rem}

\begin{thmdef}
Let $\cl$ be an openly persistent preclosure operation on ideals over a scheme $X$.  Then we can extend it to a preclosure operation on sheaves of ideals.

In particular, let $\cI \subseteq \cO_X$ be a sheaf of ideals.  Define a sheaf $\cI^\cl$ of ideals as follows: For $U \subseteq X$ open and $s \in \cO_X(U)$, we say $s \in \cI^\cl(U)$ if for any $x\in U$, there is an open affine set $V$ with $x \in V \subseteq U$ such that $s\vert_V \in \cI(V)^\cl$ in $\cO_X(V)$.

Given $s\in \cO_X(U)$, we have $s\in \cI^\cl(U)$ if and only if there is some collection $\{U_\alpha\}_{\alpha \in A}$ of open affine subsets of $U$ with $U = \bigcup_{\alpha \in A} U_\alpha$ such that $s\vert_{U_\alpha} \in \cI(U_\alpha)^\cl$ for all $\alpha \in A$.

The above definition makes $\cI^\cl$ a sheaf of ideals of $\cO_X$.  Moreover, this is a preclosure operation on sheaves of ideals, in that \begin{itemize}
    \item For any ideal sheaf $\cI$ of $\cO_X$, $\cI \subseteq \cI^\cl$, and
    \item Given an ideal sheaf inclusion $\cJ\subseteq \cI$, $\cJ^\cl$ is a subsheaf of $\cI^\cl$.
\end{itemize}
If, moreover, $X$ is Noetherian and $\cl$ is a closure operation (i.e. idempotent), then we obtain a closure operation on ideal sheaves, in that we have idempotence.  That is:
\begin{itemize}
    \item For any ideal sheaf $\cI$, we have $(\cI^\cl)^\cl = \cI^\cl$.
\end{itemize}
\end{thmdef}

\section{Glueable (pre)closure operations}\label{sec:tcsheaves}
In this section, we provide a second condition (\emph{glueability}) under which (pre)closures of sheaves are particularly well behaved.

\begin{defn}
Let $R$ be a ring and $\cl$ be an openly persistent preclosure operation on submodules (resp. ideals) over $\Spec R$.  Let $I$ be an ideal, and let $(f_\alpha)_{\alpha \in A}$ be a generating set for $I$.  Let $g\in \sqrt{I}$.  Let $L \subseteq M$ be $R$-modules (resp. let $J$ be an ideal) and $z \in M$ (resp. $z\in R$) such that for all $\alpha \in A$, we have $z/1 \in (L_{f_\alpha})^\cl_{M_{f_\alpha}}$ (resp. $z/1 \in (J_{f_\alpha})^\cl$ in $R_{f_\alpha}$).  We say $\cl$ is \emph{glueable over $R$} if in any such circumstance, it follows that $z/1 \in (L_g)^\cl_{M_g}$ (resp. $z/1 \in (J_g)^\cl$ in $R_g$).

If $\cl$ is an openly persistent preclosure operation on submodules (resp. ideals) over a scheme $X$, we say it is \emph{glueable over $X$} if for all affine open $U=\Spec R \subseteq X$, $\cl$ is glueable over $R$.
\end{defn}

\begin{prop}
Let $R$ be a ring. Let $\cl$ be an openly persistent preclosure operation on submodules (or ideals) over $\Spec R$ that commutes with localizations at single elements.  Then $\cl$ is glueable over $R$.
\end{prop}

\begin{proof}
We prove the module case; the proof of the ideal case is identical.

Let $z, f_\alpha, I, g, L,M$ be as in the definition of glueability. Then for all $\alpha$, we have $z/1 \in (L_{f_\alpha})^\cl_{M_{f_\alpha}} = (L^\cl_M)_{f_\alpha}$.  Thus, for each $\alpha$, there is some positive integer $n_\alpha$ such that $f_\alpha^{n_\alpha}z \in L^\cl_M$.  Since radicals of ideals are insensitive to powers of generators, we have $g \in \sqrt{(\{f_\alpha^{n_\alpha}\}_{\alpha \in A})}$, whence for some positive integer $N$ we have $g^N \in (\{f_\alpha^{n_\alpha}\}_{\alpha \in A})$.  Thus, there exist $\alpha_1, \ldots, \alpha_s \in A$ and $r_1, \ldots, r_s \in R$ with $g^N = \sum_{i=1}^s r_i f_{\alpha_i}^{n_{\alpha_i}}$.  Hence, $g^N z = \sum_{i=1}^s r_i \cdot \left(f_{\alpha_i}^{n_{\alpha_i}}z\right) \in L^\cl_M$, so that $z/1 \in (L^\cl_M)_g$.
\end{proof}

It is important at this point to establish some properties of glueable preclosure operations.

\begin{prop}\label{pr:genindep}
Let $\cl$ be a glueable preclosure operation on submodules over $R$.  Let $I$ be an ideal of $R$, and let $(f_\alpha)_{\alpha \in A}$, $(g_\beta)_{\beta \in B}$ be two generating sets of ideals that have the same radical as $I$.  Let $M$ be an $R$-module, $L \subseteq M$ a submodule, and $z\in M$.  Then $z/1 \in (L_{f_\alpha})^\cl_{M_{f_\alpha}}$ over the ring $R_{f_\alpha}$ for all $\alpha \in A$, if and only if $z/1 \in (L_{g_\beta})^\cl_{M_{g_\beta}}$ over the ring $R_{g_\beta}$ for all $\beta \in B$.
\end{prop}

\begin{proof}
The forward direction follows from the definition because each $g_\beta \in \sqrt{(\{f_\alpha : \alpha \in A\})}$.  The backward direction follows from reversing the roles of the $f$s and the $g$s.
\end{proof}

\begin{cor}\label{cor:genunit}
Let $\cl$ be a glueable preclosure operation on submodules over $R$.  Let $(f_\alpha)_{\alpha \in A}$ be elements of $R$ that generate the unit ideal. Let $M$ be an $R$-module, $L$ a submodule of $M$, and $z\in M$. Then $z \in L^\cl_M$ if and only if for each $\alpha$, we have $z/1 \in (L_{f_\alpha})^\cl_{M_{f_\alpha}}$.
\end{cor}

\begin{proof}
In the above proposition, let $B$ be a singleton and $g=g_\beta =1$.
\end{proof}

\begin{thm}\label{thm:glueable}
Let $\cl$ be a glueable preclosure operation on submodules over a scheme $X$.  Let $\cM$ a quasi-coherent $\cO_X$-module, and $\cL$ a quasi-coherent $\cO_X$-submodule of $\cM$.  Let $U$ be an open set in $X$ and $s \in \cM(U)$.  Then $s\in \cL^\cl_\cM(U)$ if and only if for all affine open subsets $V$ of $U$, $s\vert_V \in \cL(V)^\cl_{\cM(V)}$.

Moreover, for any affine open subset $V$ of $X$, we have $\cL^\cl_\cM(V) = \cL(V)^\cl_{\cM(V)}$.

Indeed, $\cL^\cl_\cM$ is unique with respect to this property.  That is, suppose $\cF$ is a subsheaf of $\cM$ such that $\cF(V) = \cL(V)^\cl_{\cM(V)}$ for all open affine subsets $V$ of $X$.  Then $\cF = \cL^\cl_\cM$.
\end{thm}

\begin{proof}
The ``if'' direction in the first paragraph is immediate.  So suppose $s\in \cL^\cl_\cM(U)$, and let $V$ be an affine open subset of $U$. 
Since $V$ is affine, we have $V=\Spec R$ for some $R$.  Let $\{U_\alpha\}_{\alpha \in A}$ be an affine open cover of $U$ such that $s\vert_{U_\alpha} \in \cL(U_\alpha)^\cl_{\cM(U_\alpha)}$ for all $\alpha \in A$.  For each $\alpha \in A$, since $U_\alpha \cap V$ is an open subset of $V$, one can find a positive integer $n_\alpha$ and elements $g_{\alpha,1}, \ldots, g_{\alpha, n_\alpha} \in R$ such that $U_\alpha \cap V = \bigcup_{i=1}^{n_\alpha} D_R(g_{\alpha, i})$.  For any maximal ideal $\m$ of $R$, we have $\m \in U_\alpha \cap V$ for some $\alpha$, whence there is some $g_{\alpha,i}$, $1\leq i \leq n_\alpha$, such that $\m \in D(g_{\alpha,i})$, which means that $g_{\alpha,i} \notin \m$.  It follows that the set $\{g_{\alpha,i} \mid \alpha \in A, 1\leq i \leq n_\alpha\}$ generates the unit ideal $R$.

Now let $M = \cM(V)$ and $L = \cL(V)$.  Write $z=s\vert_V \in M$.  Then for any pair $(\alpha,i)$, we have $s\vert_{D(g_{\alpha,i})} = (s\vert_{U_\alpha})\vert_{D(g_{\alpha,i})} \in \cL(D(g_{\alpha,i}))^\cl_{\cM(D(g_{\alpha,i}))}$.  But by quasi-coherence of the sheaves $\cL$ and $\cM$, we have $\cL(D(g_{\alpha,i})) = L_{g_{\alpha,i}}$, $\cM(D(g_{\alpha,i})) = M_{g_{\alpha,i}}$, and $s\vert_{D(g_{\alpha,i})} = z/1 \in M_{g_{\alpha,i}}$.  Hence, $z/1 \in (L_{g_{\alpha,i}})^\cl_{M_{g_{\alpha,i}}}$.  Since this holds for all such pairs, and since the $g_{\alpha,i}$ generate $R$, then by Corollary~\ref{cor:genunit}, we have $s\vert_V=z\in L^\cl_M = \cL(V)^\cl_{\cM(V)}$.

For the ``moreover'' statement, first let $s\in \cL^\cl_\cM(V)$.  Then since $V$ is an affine open subset of itself, we have $s = s\vert_V \in \cL(V)^\cl_{\cM(V)}$ by what we have shown above. Conversely, let $s\in \cL(V)^\cl_{\cM(V)}$. Then for any open affine subset $W$ of $V$, open persistence of $\cl$ shows that $s\vert_W \in \cL(W)^\cl_{\cM(W)}$.  Hence, $s\in \cL^\cl_\cM(V)$.

For the ``indeed'' statement, let $U$ be an open subset of $X$ and $s\in \cM(U)$.  If $s\in \cF(U)$, then by the restriction axiom, for any open affine subset $V$ of $U$, we have $s\vert_V \in \cF(V) = \cL(V)^\cl_{\cM(V)}$.  Thus, $s \in \cL^\cl_\cM(U)$.  Conversely, suppose $s\in \cL^\cl_\cM(U)$. Let $\{U_\alpha \mid \alpha \in A\}$ be an affine open cover of $U$.  Then for each $\alpha\in A$, we have $s\vert_{U_\alpha} \in \cL(U_\alpha)^\cl_{\cM(U_\alpha)} = \cF(U_\alpha)$.  Since $\cF$ is a subsheaf of $\cM$, it follows that $s\in \cF(U)$. Thus, $\cF = \cL^\cl_\cM$.
\end{proof}

One can make a similar statement for glueable preclosure operations on \emph{ideals} over a scheme, with identical proof:

\begin{thm}\label{thm:glueableideals}
Let $\cl$ be a glueable preclosure operation on ideals over a scheme $X$.  Let $\cI$ a quasi-coherent sheaf of $\cO_X$-ideals.  Let $U$ be an open set in $X$ and $s \in \cO_X(U)$.  Then $s\in \cI^\cl(U)$ if and only if for all affine open subsets $V$ of $U$, $s\vert_V \in \cI(V)^\cl$ in $\cO_X(V)$. 

Moreover, for any affine open subset $V$ of $X$, we have $\cI^\cl(V) = \cI(V)^\cl$ in $\cO_X(V)$.

Indeed, $\cI^\cl$ is unique with respect to this property.  That is, suppose $\cF$ is a sheaf of $\cO_X$-ideals such that $\cF(V) = \cI(V)^\cl$ for all open affine subsets $V$ of $X$.  Then $\cF = \cI^\cl$.
\end{thm}

\section{Quasi-coherence and  localization at elements}
This section examines conditions for quasi-coherence of (pre)closures of subsheaves, as well as conditions for subsheaves being closed.

\begin{thm}\label{thm:glueqcoh}
Let $R$ be a ring, $M$ an $R$-module, $L \subseteq M$, and $\cl$ a glueable preclosure operation on submodules over $\Spec R$
.  The following are equivalent: \begin{enumerate}[(a)]
    \item For any $f\in R$, we have $(L_f)^\cl_{M_f} = (L^\cl_M)_f$.
    \item The sheaf $\tilde{L}^\cl_{\tilde{M}}$ over $\Spec R$ is quasi-coherent.
\end{enumerate}
\end{thm}

\begin{proof}
(b) $\implies$ (a): Let $X=\Spec R$, $V = D(f)$, $\cL = \tilde L$, and $\cM = \tilde M$.  Then $\cL^\cl_\cM(V) = \cL(V)^\cl_{\cM(V)} = (L_f)^\cl_{M_f}$.  On the other hand, by quasi-coherence of $\cL^\cl_\cM$, we have $\cL^\cl_\cM(V) = \cL^\cl_\cM(X) \otimes_{\cO_X(X)} \cO_X(V) = L^\cl_M \otimes_R R_f = (L^\cl_M)_f$.

(a) $\implies$ (b): Let $X =\Spec R$.  Let $C = L^\cl_M$.  By Theorem~\ref{thm:glueable}, $\tilde L^\cl_{\tilde M}(X) = L^\cl_M$.  Since $\tilde C$ is quasi-coherent and $\tilde C(X) = L^\cl_M = \tilde L^\cl_{\tilde M}(X)$, it follows that $\tilde C$ is a subsheaf of $\tilde L^\cl_{\tilde M}$.
On the other hand, let $U$ be an open set and let $s\in \tilde L^\cl_{\tilde M}(U)$.  Then there is some open affine cover $\{U_\alpha \mid \alpha \in A\}$ of $U$ with $s\vert_{U_\alpha} \in \tilde L(U_\alpha)^\cl_{\tilde M(U_\alpha)}$.  We can refine the cover so that $U_\alpha = D(g_\alpha)$ for some elements $g_\alpha \in R$. Then by open persistence and the assumption (a), we have $s\vert_{D(g_\alpha)} \in (L_{g_\alpha})^\cl_{M_{g_\alpha}} = (L^\cl_M)_{g_\alpha} = \tilde C(D(g_\alpha))$.  Hence, $s \in \tilde C(U)$.  Therefore, $\tilde L^\cl_{\tilde M} = \tilde C$.
\end{proof}

\begin{prop}\label{pr:Lclosed}
Let $X$ be a scheme, and $\cl$ a glueable preclosure operation on modules over $X$.  Let $\cL \subseteq \cM$ be quasi-coherent $\cO_X$-modules.  The following are equivalent. \begin{enumerate}[(a)]
    \item $\cL=\cL^\cl_\cM$.
    \item There is an open affine cover $\{U_\alpha \mid \alpha \in A\}$ of $X$ such that for each $\alpha$, $\cL(U_\alpha) = \cL(U_\alpha)^\cl_{\cM(U_\alpha)}$, and $\cL^\cl_\cM$ is quasi-coherent.
    \item For any open affine subset $V$ of $X$, we have $\cL(V) = \cL(V)^\cl_{\cM(V)}$.
\end{enumerate}
If $X=\Spec R$ is affine, another equivalent condition is the following: \begin{enumerate}[({d})]
    \item For all $f\in R$, $L_f = (L_f)^\cl_{M_f}$, where $L=\cL(X)$ and $M = \cM(X)$.
\end{enumerate}
\end{prop}

\begin{proof}
(a) $\implies$ (c): By Theorem~\ref{thm:glueable}, $\cL^\cl_\cM(V) = \cL(V)^\cl_{\cM(V)}$.  But by assumption, $\cL^\cl_\cM(V) = \cL(V)$.

(c) $\implies$ (a): This follows from the last part of Theorem~\ref{thm:glueable}.

(a) $\implies$ (b): Quasi-coherence is automatic. So let $\{U_\alpha \mid \alpha \in A\}$ be any open affine cover of $X$.  Then by Theorem~\ref{thm:glueable}, for any $\alpha \in A$ we have $\cL(U_\alpha)^\cl_{\cM(U_\alpha)} = \cL^\cl_\cM(U_\alpha) = \cL(U_\alpha)$.

(b) $\implies$ (a): Let $U$ be an open subset of $X$, let $s\in \cL^\cl_\cM(U)$, and let $x\in U$.  Then there is some $\alpha \in A$ such that $x \in U_\alpha$.  Let $V$ be an open affine set such that $x \in V \subseteq U \cap U_\alpha$.  Then by the assumptions in (b) and by Theorem~\ref{thm:glueable}, we have $s\vert_V \in \cL^\cl_\cM(V) = \cL^\cl_\cM(U_\alpha)\vert_V = \cL(U_\alpha)^\cl_{\cM(U_\alpha)}\vert_V = \cL(U_\alpha)\vert_V = \cL(V)$.  Hence, $s\in \cL(U)$.  Thus, $\cL(U) = \cL^\cl_\cM(U)$.  Since $U$ was arbitrary, $\cL = \cL^\cl_\cM$.

Now we handle the special case where $X=\Spec R$ is affine.

(c) $\implies$ (d): Just apply the statement to the set $V=D(f)$.

(d) $\implies$ (b): For each $f\in R$, we have $(L_f)^\cl_{M_f} = L_f = (L^\cl_M)_f$ (the second equality by applying the assumption in (d) to $1\in R$ and then localizing at $f$).  Hence by Theorem~\ref{thm:glueqcoh}, $\cL^\cl_\cM$ is quasi-coherent.  The first statement in (b) follows from considering the open affine cover $\{D(f) \mid f\in R\}$ of $\Spec R$.
\end{proof}

\begin{prop}\label{pr:allsubsclosed}
Let $\cl$ a glueable preclosure operation on  modules over a scheme $X$. Let $\cM$ be a quasi-coherent $\cO_X$-module.
The following are equivalent: \begin{enumerate}[(a)]
    \item\label{it:clreg-qcmods} For any quasi-coherent sheaf $\cL$ of submodules of $\cM$, we have $\cL^\cl_\cM = \cL$.
    \item\label{it:clreg-modsh} For any sheaf $\cL$ of submodules of $\cM$, we have $\cL^\cl_\cM = \cL$.
    \item\label{it:clreg-mdaffines} For any affine open subset $U$ of $X$, every $\cO_X(U)$-submodule of $\cM(U)$ is $\cl$-closed.
\end{enumerate}
\end{prop}

\begin{proof}
\ref{it:clreg-modsh} $\implies$ \ref{it:clreg-qcmods}: Automatic.

\ref{it:clreg-qcmods} $\implies$ \ref{it:clreg-mdaffines}: Let $L$ be an $\cO_X(U)$-submodule of $M := \cM(U)$.  Then by \cite[Tag 01PE]{stacks-project}, there is a quasi-coherent sheaf $\cL$ of $\cO_X$-submodules of $\cM$ such that $\cL\vert_U = \tilde L$.  Then by Theorem~\ref{thm:glueable}, $L^\cl_M = \cL(U)^\cl_{\cM(U)} = \cL^\cl_\cM(U) = \cL(U) = L$. 

\ref{it:clreg-mdaffines} $\implies$ \ref{it:clreg-modsh}: Let $U$ be an open subset of $X$.  Let $s \in \cL^\cl_\cM(U)$.  Then there is some affine open cover $\{U_\alpha \mid \alpha \in A\}$ of $U$ such that for all $\alpha \in A$, we have $s\vert_{U_\alpha} \in \cL(U_\alpha)^\cl_{\cM(U_\alpha)} = \cL(U_\alpha)$. Then since $\cL$ is a subsheaf of $\cM$, it follows that $s \in \cL(U)$.  Hence $\cL^\cl_\cM(U) = \cL(U)$ for all $U$, so that $\cL^\cl_\cM = \cL$.
\end{proof}

We also note the following ideal-theoretic version, which admits precisely the same proof.

\begin{prop}
Let $\cl$ a glueable preclosure operation on ideals over a scheme $X$. 
The following are equivalent: \begin{enumerate}[(a)]
    \item\label{it:clreg-qcideals} For any quasi-coherent sheaf $\cI$ of ideals on $X$, we have $\cI^\cl = \cI$.
    \item\label{it:clreg-idsh} For any ideal sheaf $\cI$ on $X$, we have $\cI^\cl = \cI$.
    \item\label{it:clreg-affines} For any affine open subset $U$ of $X$, every ideal of $\cO_X(U)$ is $\cl$-closed.
\end{enumerate}
\end{prop}

\section{Application to tight closure and its variants}\label{sec:tc}

We start with the following definition.
\begin{defn}\label{def:psys}
Let $R$ be an $\F_p$-algebra.  A \emph{$p$-system of ideals} is a sequence of ideals $\ib_\bullet = \{\ib_{p^e}\}_{e=1}^\infty$ such that for all pairs $q,q'$ of powers of $p$, we have $\ib_q \ib_{q'}^{[q]} \subseteq \ib_{qq'}$.
\end{defn}

\begin{rem}
Let $R$ be an $\F_p$-algebra and $\ib_\bullet = \{\ib_{p^e}\}_{e=1}^\infty$ a $p$-system of ideals.  Then $\ib_\bullet$ is \begin{itemize}
    \item a \emph{$p$-family of ideals} \cite[Definition 1.1]{HeJe-Okounkov} if they satisfy the stronger condition that $\ib_q^{[p]} \subseteq \ib_{pq}$ for all powers $q$ of $p$.
    \item a \emph{$F$-graded system of ideals} \cite[Definition 3.20]{Bl-testpe} if it is descending (i.e. $\ib_q \supseteq \ib_{pq}$ for all $q$).
\end{itemize}
Hence, the notion of $p$-system is a common generalization of $p$-families and $F$-graded systems of ideals.
\end{rem}

With this setup, we can introduce the following definition.
\begin{defn}
Let $R$ be a Noetherian $\F_p$-algebra, let $\ib_\bullet$ be a $p$-system of ideals, and let $L \subseteq M$ be $R$-modules.  Then for $z\in M$, we say $z$ is in the \emph{$\ib_\bullet$-tight closure of $L$ in $M$} if there is some $c\in R^\circ$ and some power $q_0$ of $p$ such that for all $q\geq q_0$, we have \[
c \ib_q z^q_M \subseteq L^{[q]}_M.
\]
We write $z\in L^{*\ib_\bullet}_M$.
\end{defn}

This generalizes several notions of tight closure  in the literature, e.g.: \begin{itemize}
\item Setting $\ib_q := R$ for all powers $q$ of $p$, this just gives the tight closure operation of Hochster and Huneke \cite{HHmain}.
\item If $\ia_1, \ldots, \ia_n$ are ideals of $R$ such that $\ia_j \cap R^\circ \neq \emptyset$ for each $j$, and $t_1, \ldots, t_n$ are fixed positive real numbers, then setting $\ib_q := \ia_1^{\lceil t_1q\rceil} \cdots \ia_n^{\lceil t_nq\rceil}$, we recover the \emph{$\ia_1^{t_1} \cdots \ia_n^{t_n}$-tight closure} of Hara and Yoshida \cite[Remark 6.2(2)]{HaYo-atc}.
\item Given a collection $\{\ia_n\}_{n=1}^\infty$ of ideals such that $\ia_n \cap R^\circ \neq \emptyset$ for all $n$, and such that $\ia_n \ia_m \subseteq \ia_{n+m}$ (i.e. a \emph{graded family of ideals} in the terminology of \cite{EiLaSm-symbolic}), if we set $\ib_q = \ia_q$ for each power $q$ of $p$, then the $\ib_\bullet$-tight closure coincides with the \emph{$||\ia_\bullet||$-tight closure} of Hara \cite[Definition 2.7]{Ha-multanalog}
\end{itemize}

\begin{lemma}\label{lem:pfpreclosure}
    Let $R$ be a Noetherian $\F_p$-algebra and $\ib_\bullet$ a $p$-system of ideals.  Then $\ib_\bullet$-tight closure is a preclosure operation on $R$-submodules.
\end{lemma}

\begin{proof}
Let $K \subseteq L \subseteq M$ be $R$-modules.  It is clear from the definition that $L \subseteq L^{*\ib_\bullet}_M \subseteq M$ and that $K^{*\ib_\bullet}_M \subseteq L^{*\ib_\bullet}_M$. We need only show that $L^{*\ib_\bullet}_M$ is a submodule of $M$.

Accordingly, let $y,z \in L^{*\ib_\bullet}_M$, and $r\in R$.  Then there exist powers $q_0, q_1$ of $p$ and $c,d \in R^\circ$ such that whenever $q \geq q_0$ (resp. $q \geq q_1$), we have $c \ib_q y^q_M \subseteq L^{[q]}_M$ (resp. $d \ib_q z^q_M \subseteq L^{[q]}_M$).  Hence for all $q \geq \max\{q_0, q_1\}$, we have $cd \ib_q (y+rz)^q_M \subseteq d \cdot (c \ib_q y^q_M) + cr^q \cdot (d \ib_q z^q_M) \subseteq L^{[q]}_M$.  Thus $y+rz \in L^{*\ib_\bullet}_M$.
\end{proof}

To extend this definition to sheaves of submodules, we need the following result.

\begin{prop}\label{pr:pfamaffine}
Let $R$ be a Noetherian $\F_p$-algebra, and let $\ib_\bullet = \{\ib_{p^e}\}_{e=1}^\infty$ be a collection of ideals indexed by the powers of $p$. For any affine open subset $U$ of $\Spec R$, let $R(U) := \cO_{\Spec R}(U)$ be the corresponding ring, with $R$-algebra structure induced from the open immersion $U \into \Spec R$, and let $\ib(U)_\bullet = \{\ib(U)_{p^e}\}_{e=1}^\infty$ be the family of $R(U)$-ideals given by $\ib(U)_q := \ib_q R(U)$.  The following are equivalent: \begin{enumerate}[(a)]
\item\label{it:pfam} $\ib_\bullet$ is a $p$-system of ideals of $R$.
\item\label{it:pfamcover} There is an affine open cover $\{U_\alpha\}_{\alpha \in \Lambda}$ of $\Spec R$ such that $\ib(U_\alpha)_\bullet$ is a $p$-system of $R(U_\alpha)$-ideals for each $\alpha \in \Lambda$.
\item\label{it:pfamopen} For any affine open subset $U$ of $\Spec R$, $\ib(U)_\bullet$ is a $p$-system of ideals of $R(U)$.
\end{enumerate}
\end{prop}

\begin{proof}
\ref{it:pfam} $\implies$ \ref{it:pfamopen}: For any affine open subset $U$ of $\Spec R$,  \[\ib(U)_q \ib(U)_{q'}^{[q]} = 
(\ib_q R(U)) (\ib_{q'} R(U))^{[q]} =
\ib_q \ib_{q'}^{[q]} R(U)  \subseteq \ib_{qq'} R(U) = \ib(U)_{qq'}.\]

\ref{it:pfamopen} $\implies$ \ref{it:pfamcover}: Choose the cover consisting of all affine open sets.

\ref{it:pfamcover} $\implies$ \ref{it:pfam}: Refine $\{U_\alpha\}$ to an open cover of the form $\{D(f_\beta)\}_{\beta \in \Sigma}$.  Then whenever we have $D(f_\beta) \subseteq U_\alpha$, it follows that for any $\beta \in \Sigma$ and any powers $q,q'$ of $p$ that \begin{align*}
\ib(D(f_\beta))_q \ib(D(f_\beta))_{q'}^{[q]} &=(\ib_q R_{f_\beta})(\ib_{q'} R_{f_\beta})^{[q]} = (\ib_q R(U_\alpha))( \ib_{q'}^{[q]} R(U_\alpha))R_{f_\beta} \\
&= \ib(U_\alpha)_q \ib(U_\alpha)_{q'}^{[q]} R_{f_\beta} \subseteq \ib(U_\alpha)_{qq'} R_{f_\beta}\\
&= (\ib_{qq'} R(U_\alpha)) R_{f_\beta} = \ib_{qq'} R_{f_\beta} = \ib(D(f_\beta))_{qq'}.
\end{align*}
Hence, $\ib(D(f_\beta))_\bullet$ is a $p$-system of ideals in $R_{f_\beta}$ for each $\beta \in \Sigma$.

Also, since the $f_\beta$ generate $R$, there exist $\beta_1, \ldots, \beta_t \in \Sigma$ such that $R= (f_{\beta_1}, \ldots, f_{\beta_t})$.  Now let $q,q'$ be powers of $p$ and let $x\in \ib_q$ and $y \in \ib_{q'}$, so that for each $1 \leq j\leq t$, $\frac x1 \in \ib(D(f_{\beta_j}))_q$ and $\frac y1 \in \ib(D(f_{\beta_j}))_{q'}$.  Then we have $\frac{xy^q}1 =(x/1) (y/1)^q \in \ib(D(f_{\beta_j}))_{qq'}=\ib_{qq'} R_{f_{\beta_j}}$.  Thus there is some $m_j \in \N$ with $f_{\beta_j}^{m_j} xy^q \in \ib_{qq'}$.  But since the $f_{\beta_j}^{m_j}$ generate $R$, there exist $a_1, \ldots, a_t \in R$ with $\sum_{j=1}^t a_j f_{\beta_j}^{m_j} = 1$.  Hence, $xy^q = \sum_{j=1}^t a_j f_{\beta_j}^{m_j} xy^q \in \ib_{qq'}.$
\end{proof}


\begin{defpr}
Let $X$ be a locally Noetherian $\F_p$-scheme, and let $\cb_\bullet := \{\cb_{p^e}\}_{e=1}^\infty$ be a sequence of quasi-coherent ideal sheaves over $X$.  For any affine open subset $U$ of $X$, set $\cb(U)_\bullet := \{\cb_{p^e}(U)\}_{e=1}^\infty$. The following are equivalent: \begin{enumerate}[(a)]
    \item There is an affine open cover $\{U_\alpha\}$ of $X$ such that $\cb(U_\alpha)_\bullet$ is a $p$-system of $\cO_X(U_\alpha)$-ideals for each $\alpha \in \Lambda$.
    \item For any affine open subset $U$ of $X$, $\cb(U)_\bullet$ is a $p$-system of $\cO_X(U)$-ideals.
\end{enumerate}
Under these conditions we call $\cb_\bullet$ a \emph{$p$-system of ideal sheaves over $X$}.
\end{defpr}

\begin{proof}
We need only prove (a) $\implies$ (b).  For each $\alpha \in \Lambda$, let $\{V_\beta\}_{\beta \in \Sigma_\alpha}$ be an affine open cover of $U \cap U_\alpha$.  By Proposition~\ref{pr:pfamaffine}(\ref{it:pfam} $\Rightarrow$ \ref{it:pfamopen}), $\cb(V_\beta)_\bullet$ is a $p$-system of $\cO_X(V_\beta)$-ideals for all $\beta \in \Sigma_\alpha$ for all $\alpha \in \Lambda$.  But then by Proposition~\ref{pr:pfamaffine}(\ref{it:pfamcover} $\Rightarrow$ \ref{it:pfam}), $\cb(U)_\bullet$ is a $p$-system of $\cO_X(U)$-ideals.
\end{proof}

\begin{defn}
Let $X$ be a locally Noetherian $\F_p$-scheme. 
 Let $\cb_\bullet = \{\cb_{p^e}\}_{e=1}^\infty$ be a $p$-system of ideal sheaves over $X$.  For any affine open subset $U \subseteq X$ and any $R$-submodule inclusion $L \subseteq M$ (where $R = \cO_X(U))$, set $\cb(U)_\bullet := \{\cb_{p^e}(U)\}_{e=1}^\infty$ and $L^{*\cb_\bullet}_M := L^{*\cb(U)_\bullet}_M$.
%
%
\end{defn}

\begin{thm}\label{thm:pfglue}
Let $X$ be a locally Noetherian $\F_p$-scheme.
Let $\cb_\bullet$ be a $p$-system of ideal sheaves over $X$.  Then $*\cb_\bullet$ is a glueable preclosure operation on submodules over $X$.
\end{thm}

We know from Lemma~\ref{lem:pfpreclosure} that whenever $\ib_\bullet$ is a $p$-system of ideals in a Noetherian ring $R$ of characteristic $p$, $*\ib_\bullet$ is a preclosure operation.  It remains to show open persistence and glueability. We begin with a lemma that is presumably (see \cite[parenthetical comment within the statement of Theorem 6.22]{HHbase}) well-known.

\begin{lemma}\label{lem:percirc}
Let $\phi: R \ra S$ be a homomorphism of commutative rings that satisfies going-down.  Then $\phi(R^\circ) \subseteq S^\circ$. In particular, this holds if $\phi$ is flat.
\end{lemma}

\begin{proof}
Let $c\in R^\circ$.  Let $Q$ be a prime ideal of $S$ with $\phi(c) \in Q$.  Then $c \in \q := \phi^{-1}(Q)$, so by assumption, $\q$ cannot be a minimal prime of $R$.  Let $\p$ be a prime ideal of $R$ properly contained in $\q$.  Then since $\phi$ satisfies the going-down property, there is some prime ideal $P$ of $S$ with $P \subseteq Q$ and $\phi^{-1}(P) = \p$.  Since $\p \neq \q$, it follows that $P \neq Q$.  Thus, $Q$ is not a minimal prime of $S$.  Since $Q \in \Spec S$ was arbitrary with $\phi(c) \in Q$, it follows that $\phi(c) \in S^\circ$.
The final statement follows from \cite[Theorem 9.5]{Mats}.
\end{proof}

Next, we prove open persistence of the operation.

\begin{lemma}\label{lem:perspf}
Let $X$ be a locally Noetherian $\F_p$-scheme. 
 Let $\cb_\bullet$ be a $p$-system of ideal sheaves over $X$.  Then $*\cb_\bullet$ is openly persistent on submodules over $X$.
%
\end{lemma}

\begin{proof}
Let $V \subseteq U$ be affine open subsets of $X$.  Write $R = \cO_X(U)$, $S := \cO_X(V)$, and $\phi: R \ra S$ the corresponding ring map.  For each power $q=p^e$ of $p$, write $\ib_q = \cb_q(U)$. Let $L \subseteq M$ be $R$-modules and $x \in L^{*\cb_\bullet}_M = L^{*\ib_\bullet}_M$.  Then there is some $c\in R^\circ$ and some power $q_0$ of $p$ such that for all $q\geq q_0$, $c \ib_q x^q_M \subseteq L^{[q]}_M$. Applying $\phi$ and $(-) \otimes_R S$, it follows that $\phi(c) (\ib_q S) (x \otimes 1)^{[q]}_{M \otimes S} \subseteq (LS)^{[q]}_{M \otimes_R S}$ for all $q\geq q_0$.  Since $\phi(R^\circ) \subseteq S^\circ$ (by virtue of the flatness of the open immersion $V \subseteq U$ and Lemma~\ref{lem:percirc}; see also Remark~\ref{rem:flatly}), we have $\phi(c) \in S^\circ$.  Moreover, the quasi-coherence of the ideal sheaves $\cb_q$ implies that $\cb_q(V) = \ib_qS$ for each power $q$ of $p$.  Thus, letting $\ib_\bullet S := \{\ib_{p^e} S\}_{e=1}^\infty$, we have $x \otimes 1 \in (LS)^{*\ib_\bullet S}_{M \otimes S} = (LS)^{*\cb_\bullet}_{M \otimes_RS}$.
\end{proof}

\begin{proof}[Proof of Theorem~\ref{thm:pfglue}]
We may assume $X$ is affine. Write $R = \cO_X(X)$.  For each power $q$ of $p$, $\ib_q := \cb_q(X)$.
Let $(f_\alpha)_{\alpha \in A}$ be elements of $R$, with $I$ the ideal they generate, and let $g\in \sqrt I$.  Let $L \subseteq M$ be $R$-modules and $z\in M$ such that $z/1 \in (L_{f_\alpha})^{*\ib_\bullet R_{f_\alpha}}_{M_{f_\alpha}}$ for each $\alpha \in A$, where $\ib_\bullet S := \{\ib_{p^e}S\}_{e=1}^\infty$ for any $R$-algebra $S$.  Since $g\in \sqrt{I}$, there exist $\alpha_1, \ldots, \alpha_k \in A$ such that $g \in \sqrt{(f_{\alpha_1}, \ldots, f_{\alpha_k})}$.  For each $1\leq i \leq k$, there is some power $q_i$ of $p$ and some $c_i \in R^\circ$ such that $(c_i \ib_q z^q_M)_{f_{\alpha_i}} \subseteq (L_{f_{\alpha_i}})^{[q]}_{M_{f_{\alpha_i}}} = (L^{[q]}_M)_{f_{\alpha_i}}$ for all $q \geq q_i$.  Let $q_0 := \max\{q_1, \ldots, q_k\}$ and $c = \prod_{i=1}^k c_i$.  Then there exist positive integers $N_{i,q}$ such fthat for all $i$ and $q\geq q_0$, $f_{\alpha_i}^{N_{i,q}} c \ib_q z^q_M \subseteq L^{[q]}_M$.

Fix some $q\geq q_0$.  Since $g \in \sqrt{(f_{\alpha_i}^{N_{i,q}} \mid 1\leq i \leq k)}$, there exist $r_1, \ldots, r_k \in R$ and $u\in \N$ such that $g^u = \sum_{i=1}^k r_i f_{\alpha_i}^{N_{i,q}}$.  Hence $g^u  c \ib_q z^q_M = \sum_{i=1}^k r_i \cdot \left( f_{\alpha_i}^{N_{i,q}}  c \ib_q z^q_M\right)\subseteq L^{[q]}_M$.  Thus, $(c/1) (\ib_qR_g) (z/1)^q_{M_g} = \left( c \ib_q z^q_M\right)_g \subseteq (L^{[q]}_M)_g = (L_g)^{[q]}_{M_g}$.  As this holds for all $q\geq q_0$, we have $z/1 \in (L_g)^{*\ib_\bullet R_g}_{M_g}$.  But by quasi-coherence of the ideal sheaves $\cb_q$, we have $\cb_q(D(f_\alpha)) = \ib_q R_{f_\alpha}$ for all $\alpha$, and $\cb_q(D(g)) = \ib_q R_g$.  The result follows.
\end{proof}



It is instructive to compare the following corollary with \cite[Theorem 8.14]{HHmain}, where under stable test element conditions, a similar thing is proved for localizations at maximal ideals (i.e. at the stalk level).  In contrast, the following corollary has no need even of weak test elements.

\begin{cor}\label{cor:genunittc}
Let $R$ be a Noetherian ring of positive prime characteristic.  Let $(f_\alpha)_{\alpha \in A}$ be elements of $R$ that generate the unit ideal. Let $M$ be an $R$-module, $L$ a submodule of $M$, and $z\in M$. Then $z \in L^*_M$ if and only if for each $\alpha$, we have $z/1 \in (L_{f_\alpha})^*_{M_{f_\alpha}}$.
\end{cor}

\begin{mth}\label{mth} All the results from the previous sections of the paper apply to tight closure and its variants.
\end{mth}

\begin{rem}
In particular (using ordinary tight closure for simplicity), we define the tight closure of a subsheaf $\cL$ of a quasi-coherent sheaf $\cM$ by saying that a section $s \in \cM(U)$ is in $\cL^*_\cM(U)$ if $U$ has an affine open cover $\{U_\alpha\}_{\alpha \in A}$ such that $s\vert_{U_\alpha} \in \cL(U_\alpha)^*_{\cM(U_\alpha)}$ as $\cO_X(U_\alpha)$-modules for all $\alpha \in A$.  Then we see above that this gives a sheaf of submodules of $\cM$, and defines a closure operation on sheaves of submodules of $\cM$ (or in the case of most $\ib_\bullet$-tight closures, only a preclosure operation). In the case where $\cL$ is a \emph{coherent} sheaf of submodules, we have that $s \in \cL^*_\cM(U)$ if and only if for \emph{every} affine open $V$ in $U$, we have $s\vert_V \in \cL(V)^*_{\cM(V)}$ -- and $\cL^*_\cM$ is the unique sheaf that admits this property.
\end{rem}

Consider the question: Does tight closure commute with localization at single elements?  This seems to be open, as it was when Brenner and Monsky gave the counterexample to tight closure commuting with \emph{arbitrary} localization \cite[p. 572]{BM-unloc}. We now reformulate the question as follows:

\begin{thm}
Let $R$ be a Noetherian ring containing $\F_p$, and let $I$ be an ideal of $R$.  The following are equivalent \begin{enumerate}[(a)]
\item For any $f\in R$, $(I_f)^* = (I^*)_f$.
\item The sheaf $(\tilde I)^*$ of $\cO_{\Spec R}$-ideals is coherent.
\end{enumerate}
Thus, tight closure of ideals in $R$ commutes with localization at single elements if and only if the tight closure of any coherent sheaf of $\cO_{\Spec R}$-ideals is coherent.

Similarly, if $M$ is an $R$-module and $L$ a submodule, the following are equivalent: \begin{enumerate}[(a)]
\item For any $f\in R$, $(L_f)^*_{M_f} = (L^*_M)_f$.
\item The sheaf $\tilde L^*_{\tilde M}$ of $\cO_X$-submodules of $\tilde M$ is quasi-coherent.
\end{enumerate}
Thus, tight closure of submodules of $R$-modules (resp. finite $R$-modules) commutes with localization at single elements if and only if the tight closure of any quasi-coherent sheaf of submodules of a quasi-coherent (resp. coherent) sheaf of $\cO_{\Spec R}$-modules is always quasi-coherent.
\end{thm}

\begin{proof}
This follows from Theorem~\ref{thm:pfglue} (with $n=0$) and Theorem~\ref{thm:glueqcoh}.
\end{proof}

We introduce a new (non-local!) F-singularity class in the following theorem.

\begin{thm}\label{thm:semiF}
Let $R$ be a Noetherian ring of positive prime characteristic and $X=\Spec R$.  The following are equivalent: \begin{enumerate}[(a)]
    \item For any ideal $I$ of $R$, we have $\cI^*=\cI$, where $\cI = \tilde I$.
    \item For any ideal sheaf $\cI$ on $X$, we have $\cI^*=\cI$.
    \item For any affine open subset $U$ of $X$, every ideal in $\cO_X(U)$ is tightly closed.
    \item For any $f\in R$ and any ideal $I$ of $R$, we have $I_f = (I_f)^*$ (i.e. $R_f$ is weakly F-regular).
    \item For any $R$-module inclusion $L \subseteq M$, with $M$ finitely generated, we have $\cL^*_\cM = \cL$, where $\cL = \tilde L$ and $\cM = \tilde M$.
    \item For any coherent sheaf $\cM$ of $\cO_X$-modules and any sheaf $\cL$ of submodules, we have $\cL^*_\cM = \cL$.
    \item For any affine open subset $U$ of $X$ and any finite $\cO_X(U)$-module $M$, $0$ is tightly closed in $M$.
    \item For any $f\in R$ and any inclusion of finite $R$-modules $L \subseteq M$, we have $(L_f)^*_{M_f} = L_f$ as $R_f$-modules.
\end{enumerate}
\end{thm}

\begin{proof}
The equivalence of (a)--(d) (resp. (e)--(h)), follows from Proposition~\ref{pr:Lclosed} (resp. Proposition~\ref{pr:allsubsclosed}) combined with Theorem~\ref{thm:pfglue}.  It is clear that (e) $\implies$ (a) from the special case $M=R$.  Finally, it follows from \cite[Proposition 8.7]{HHmain} that (d) $\implies$ (h). 
\end{proof}

\begin{defn}
If a ring $R$ satisfies the equivalent conditions of Theorem~\ref{thm:semiF}, we say $R$ is \emph{semi F-regular}.  If, more generally, $X$ is a locally Noetherian $\F_p$-scheme, we say it is \emph{semi F-regular} if for each affine open subset $U$ of $X$, $\cO_X(U)$ is semi-F-regular.
\end{defn}

\begin{thm}
Let $X$ be a Noetherian $\F_p$-scheme.  Then $X$ is semi F-regular if and only if there is some cover $\{U_\alpha\}_{\alpha \in A}$ by open affines such that $\cO_X(U_\alpha)$ is semi F-regular for all $\alpha \in A$.
\end{thm}

\begin{proof}
The ``only if'' direction is clear from the definition.  Conversely, suppose such a cover $\{U_\alpha\}_{\alpha \in A}$ exists.  Let $V$ be an open affine subset of $X$, let $I$ be an ideal of $\cO_X(V)$.  By \cite[Tag 01PE]{stacks-project} there is a coherent sheaf $\cI$ of $\cO_X$-ideals such that $\cI\vert_V = \tilde I$.  Let $s \in I^*$.  We have $I^* = \tilde I(V)^* = \tilde I^*(V) = \cI^*(V)$ by the second part of Theorem~\ref{thm:glueable}.
For any $x\in V$, there is some $\alpha \in A$ with $x \in U_\alpha$.  Let $W$ be an affine open with $x\in W \subseteq V \cap U_\alpha$.  Then $s\vert_W \in \cI(W)^* = \cI(W)$ by Theorem~\ref{thm:semiF}, since $W \subseteq U_\alpha$ is open affine and $\cO_X(U_\alpha)$ is semi F-regular.  Since $x\in V$ was arbitrary, it follows from the gluing axiom that $s \in \cI(V) = I$.  Hence $I$ is tightly closed, so $X$ is semi F-regular.
\end{proof}

\begin{prop}\label{pr:Fregimplications}
Let $R$ be a Noetherian ring containing $\F_p$.  If $R$ is F-regular, it is semi F-regular.  If $R$ is semi F-regular, it is weakly F-regular. 
\end{prop}

\begin{proof}
If $R$ is F-regular, then all localizations of $R$ are weakly F-regular, in particular those of the form $R_f$.  Hence $R$ is semi F-regular.

On the other hand, if $R$ is semi F-regular, then $R_f$ is weakly F-regular for all $f \in R$.  In particular, $1\in R$, so $R=R_1$ is weakly F-regular.
\end{proof}

Recall that weak F-regularity is a local property on closed points.  That is: \begin{prop}[{\cite[Corollary 4.15]{HHmain}}]\label{pr:wFrlocal}
Let $R$ be a Noetherian $\F_p$-algebra.  Then $R$ is weakly F-regular if and only if for every maximal ideal $\m$ of $R$, $R_\m$ is weakly F-regular.
\end{prop}

It is natural to ask: is semi F-regularity equivalent to weak F-regularity?  This is true for \emph{Jacobson} rings.

\begin{prop}\label{pr:semiFJacob}
Let $R$ be a Noetherian $\F_p$-algebra which is a Jacobson ring.  Then $R$ is semi F-regular if and only if it is weakly F-regular.
\end{prop}

\begin{proof}
The ``only if'' direction is part of Proposition~\ref{pr:Fregimplications}.  For the converse, assume $R$ is weakly F-regular, let $f\in R$ be a non-nilpotent element (else $R_f = 0$ is weakly F-regular, vacuously) and let $\n$ be a maximal ideal of $R_f$.  Since $R$ is Jacobson, $\n = \m R_f$ for some maximal ideal $\m$ of $R$ with $f\notin \m$ \cite[Tag 00G6]{stacks-project}.  Thus, $(R_f)_\n = (R_f)_{\m R_f} = R_\m$, which is weakly F-regular since $R$ is by Proposition~\ref{pr:wFrlocal}.  As $\n$ was an arbitrary maximal ideal of $R_f$, it follows (again from Proposition~\ref{pr:wFrlocal}) that $R_f$ is weakly F-regular.  As $f\in R$ was arbitrary, $R$ is semi F-regular.
\end{proof}

On the other hand, if the two conditions were known to be equivalent for \emph{local} rings, then a major conjecture would be solved
:

\begin{prop}\label{pr:wFregsemi}
Suppose that for all Noetherian weakly F-regular local $\F_p$-algebras are semi F-regular.  Then weak F-regularity is equivalent to F-regularity for all Noetherian $\F_p$-algebras.
\end{prop}

\begin{proof}
Suppose $R$ is a weakly F-regular Noetherian $\F_p$-algebra that is not F-regular.  Then there is some prime ideal $\p$ such that $R_\p$ is not weakly F-regular; choose $\p$ to be maximal in the set of all such prime ideals.  By Proposition~\ref{pr:wFrlocal} $\p$ cannot be a maximal ideal.  Choose a prime ideal $\q \supseteq \p$ such that $\hgt \q/\p =1$.  Then $R_\q$ is weakly F-regular.  Choose $f \in \q \setminus \p$.  Then $\p (R_\q)_f$ is a maximal ideal of the ring $(R_\q)_f$.  Moreover, since we are assuming all weakly F-regular local rings are semi F-regular, we have that $R_\q$ is semi F-regular, which implies that $(R_\q)_f$ is weakly F-regular.  By Proposition~\ref{pr:wFrlocal}, we then have that $((R_\q)_f)_{\p (R_\q)_f}\cong R_\p$ is weakly F-regular, which is a contradiction.
\end{proof}

\section*{Acknowledgment}
Thanks to Holger Brenner, Craig Huneke, Karl Schwede, and Karen Smith for early conversations regarding the idea for this paper.  Thanks also to Evan Houston for Example~\ref{ex:Houston}, and to Wenliang Zhang for the idea behind Proposition~\ref{pr:wFregsemi}.  Finally, I would like to thank the anonymous referee for suggesting that in Section~\ref{sec:tc}, I should generalize from $\ia^t$-type tight closure to tight closure in the more general settings of $p$-families and $F$-graded systems of ideals.

\providecommand{\bysame}{\leavevmode\hbox to3em{\hrulefill}\thinspace}
\providecommand{\MR}{\relax\ifhmode\unskip\space\fi MR }
\providecommand{\MRhref}[2]{%
  \href{http://www.ams.org/mathscinet-getitem?mr=#1}{#2}
}
\providecommand{\href}[2]{#2}

\end{document}